%% file: SCL_paper.tex
\documentclass[preprint,12pt]{elsarticle}

\usepackage{natbib}
\usepackage{latexsym}
\usepackage{verbatim}
\usepackage[english]{babel}
\usepackage[latin1]{inputenc}
\usepackage{enumitem}
\usepackage{amsmath,amsthm}
\usepackage{mathtools}

\usepackage{algpseudocode}
\usepackage{xpatch}
\algrenewcommand\alglinenumber[1]{{\sffamily\footnotesize#1:}}

\usepackage{multirow}
\usepackage{enumitem}

\usepackage[usenames, dvipsnames]{color}

\usepackage{tikz}
\usetikzlibrary{tikzmark}
\usepackage{amsfonts}
\usepackage{amssymb}
\usepackage{amsbsy}

\usepackage{bm}

    \let\leq\leqslant
    \let\geq\geqslant
    \let\emptyset\varnothing

\input{ka-newcommands}

\input{ka-newcommands-extra}

\newcounter{todocounter}
\usepackage{todonotes}

\newtheorem{theorem}{Theorem}
\newtheorem{lemma}[theorem]{Lemma}
\newtheorem{assumption}[theorem]{Assumption}
\newtheorem{corollary}[theorem]{Corollary}
\newtheorem{remark}[theorem]{Remark}
\newcommand{\brem}{\begin{remark}}
\newcommand{\erem}{\end{remark}}

\newtheorem{definition}[theorem]{Definition}
\newtheorem{example}[theorem]{Example}
\newtheorem{proposition}[theorem]{Proposition}
\newtheorem{problem}[]{Problem}
\newtheorem{conjecture}[theorem]{Conjecture}

\newcommand{\benabc}{\begin{enumerate}[label={\em(\alph*)}, ref=(\alph*)]}
\newcommand{\beni}{\begin{enumerate}[label={\em(\roman*)}, ref=(\roman*)]}
\newcommand{\benabcrm}{\begin{enumerate}[label=(\alph*), ref=(\alph*)]}

\journal{Systems and Control Letters}

\begin{document}
\setcounter{MaxMatrixCols}{20}

\begin{frontmatter}

\title{The shortest experiment for linear system identification}

\author[1]{M.K. Camlibel}
\author[1]{H.J. van Waarde\fnref{funding}}
\affiliation[1]{organization={Bernoulli Institute for Mathematics, Computer Science and Artificial Intelligence, University of Groningen},
            country={The Netherlands},
            email={ (m.k.camlibel@rug.nl, h.j.van.waarde@rug.nl)}}
\fntext[funding]{Henk van Waarde acknowledges financial support by the Dutch Research Council (NWO) under the Talent Programme Veni Agreement (VI.Veni.22.335).}
\author[2]{P. Rapisarda} 


            \affiliation[2]{organization={School of Electronics and Computer Science, University of Southampton},
            country={United Kingdom},
            email={ (pr3@ecs.soton.ac.uk)}}

\begin{abstract}

This paper is concerned with the following problem: given an upper bound of the state-space dimension and lag of a linear time-invariant system, design a sequence of inputs so that the system dynamics can be recovered from the resulting input-output data. As our main result we propose a new online experiment design method, meaning that the selection of the inputs is iterative and guided by data samples collected in the past. We show that this approach leads to the shortest possible experiments for linear system identification. In terms of sample complexity, the proposed method outperforms offline methods based on persistency of excitation as well as existing online experiment design methods.
\end{abstract}




\end{frontmatter}



\section{Introduction}

\noindent 
\emph{Background}: In the context of system identification, \emph{experiment design} is concerned with the selection of inputs of a dynamical system in such a way that the resulting input-output data contain sufficient information about the system dynamics \citep{Goodwin1977}. Experiment design is a classical topic that has been investigated from different angles throughout the years \citep{Gevers1986,Jansson2005,Bombois2006,Iannelli2021}. 

An experiment design result that has recently been popularized is the so-called \emph{fundamental lemma} by Willems and his coauthors \cite{Willems05}. Roughly speaking, the result says that the dynamics of a linear time-invariant system can be uniquely identified from input-output data if the input data are chosen to be sufficiently persistently exciting. The fundamental lemma also provides a parameterization of all finite trajectories of the system, in terms of a data Hankel matrix. This parameterization has been applied in several recent data-driven analysis and control techniques ranging from simulation and linear quadratic regulation \cite{Markovsky2008} to predictive control \cite{Coulson19,Berberich21}, stabilization \cite{DePerzik20} and dissipativity analysis \cite{Koch2022}.

The recent interest in data-driven control has also led to extensions of the fundamental lemma itself. Its original proof was presented in the language of behavioral theory; an alternative proof for state space systems was provided in \cite{Waarde20}. Generalizations to uncontrollable systems are presented in \cite{Markovsky2023,Yue2021,Mishra2021} and extensions to continuous-time systems in \cite{Lopez2022,Rapisarda2023,Rapisarda2023b,Lopez2024}. Robust/quantitative versions are explored in \cite{Coulson2022,Berberich2023,Coulson2023} while frequency domain formulations have been considered in \cite{Ferizbegovic2021,Meijer2023}. Furthermore, the fundamental lemma has been generalized to various other model classes such as descriptor systems \cite{Schmitz2022}, flat nonlinear systems \cite{Alsalti2021}, linear
parameter-varying systems \cite{Verhoek2021}, and stochastic ones \cite{Faulwasser2023}.

The assumption of persistency of excitation, central to the fundamental lemma, imposes a lower bound on the required number of data samples for system identification. This lower bound is conservative in the sense that system identification is possible using less data, as long as certain conditions are met. Until recently, the precise conditions enabling system identification were still missing from the literature. However, in \cite{CamlibelRapisarda2024}, necessary and sufficient conditions are provided under which the data contain sufficient information for system identification, assuming minimality of the data-generating system and upper (and lower) bounds on its lag and state-space dimension. As is quite reminiscent of the subspace identification literature \cite{vanOverschee1996,Moonen1989}, these conditions involve rank properties of data Hankel matrices. However, a truly remarkable feat of \cite{CamlibelRapisarda2024} is that the \emph{depth} of these Hankel matrices is not given a priori but is \emph{determined from the data}. 

In this paper, we will build on the framework of \cite{CamlibelRapisarda2024}. However, unlike \cite{CamlibelRapisarda2024} that focused on analyzing informativity of given data sets, the purpose of this paper is to \emph{design} experiments that are informative for system identification. The main questions of this paper are therefore as follows: \vspace{-4mm}

\begin{center}
\fbox{\begin{minipage}{\textwidth}
What is the shortest possible sequence of input-output data that enables linear system identification? \\
And, if possible, how can we design the inputs of the system in order to generate such a shortest experiment? 
\end{minipage}}
\end{center} 

Apart from being of theoretical interest, these questions are also highly relevant for control applications. In fact, the computation time of, for example, data-driven predictive control scales cubically with the data length \cite{Coulson19,Fabiani2021}. Working with shorter yet informative experiments reduces the computational burden of these methods. 

\vspace{2mm}
\noindent
\emph{Contributions}: The main contributions of the paper are as follows: 
\vspace{-1mm}
\begin{itemize}
\item We assume that the unknown data-generating system is minimal and its lag and state-space dimension are upper bounded by $L$ and $N$. Then, based on \cite{CamlibelRapisarda2024}, we formulate a lower bound $T$ on the data length of any experiment that is informative for system identification (see \eqref{defT}). \vspace{-1mm}
\item We propose the experiment design method \textsc{OnlineExperiment}$(L,N)$. This procedure designs the inputs \emph{online}, i.e., on the basis of past input-output samples.  In Theorem~\ref{t:mainresult} we prove that this method leads to informative experiments of length $T$.
\end{itemize}

A remarkable outcome of this paper is the fact that experiments can be designed of length \emph{precisely equal to the lower bound $T$}. Interestingly, this number of samples $T$ depends on the unknown system and is thus not given a priori. It is revealed after the experiment design algorithm terminates. 

\vspace{2mm}
\noindent
\emph{Relation to previous work}: The online experiment design method of this paper always requires less samples than the persistency of excitation requirement of the fundamental lemma. In many cases, there is a substantial difference in number of data samples. We provide an example of this in Section~\ref{sec:comparison}. The online experiment design method of \cite{Waarde21} keeps the depth of the data Hankel matrix \emph{fixed} during the operation of the algorithm. In contrast to \cite{Waarde21}, our approach \emph{adapts the depth} of the Hankel matrix during operation of the algorithm. The rationale is that the ``correct" depth of the Hankel matrix for the necessary and sufficient conditions of \cite{CamlibelRapisarda2024} is \emph{a priori unknown}. We show in Section~\ref{sec:comparison} that our method outperforms \cite{Waarde21} in some situations, depending on the given $L$ and $N$. In other cases, our results prove that the approach of \cite{Waarde21} leads to the shortest experiments for linear system identification.
 
\vspace{2mm}
\noindent
\emph{Outline}: In Section~\ref{sec:notation} we treat preliminaries. In Section~\ref{sec:fl} we recap the fundamental lemma. In Section~\ref{sec:inf} we recall the definition and characterization of \emph{informativity for system identification} \citep{CamlibelRapisarda2024}. Subsequently, in Section~\ref{sec:problemstatement} we formalize the problem and in Section~\ref{sec:experimentdesign} we present our main results. In Section~\ref{sec:example} we provide examples, while Section~\ref{sec:comparison} contains a comparison to the state-of-the-art. Finally, in Section~\ref{sec:conclusions} we conclude the paper.

\section{Notation and preliminaries}\label{sec:notation}

We denote the $n\times n$ identity matrix by $I_n$ and the $m\times n$ zero matrix by $0_{m,n}$. The column vector containing $n$ zeros is denoted by $0_n$. For partitioned matrices containing zero and/or identity submatrices, we do not indicate the sizes of blocks that follow from their positions. The left kernel of a matrix $M \in \mathbb{R}^{m \times n}$ is defined as $\lk M := \{x \in \mathbb{R}^m \mid x^\top M = 0\}$.

A subset $\calA \subseteq \mathbb{R}^n$ is called \emph{affine} if it can be expressed as $\calA = \{x\} + \calS$ where $x \in \mathbb{R}^n$ and $\calS \subseteq \mathbb{R}^n$ is a subspace. The dimension of $\calA$ is defined as the dimension of $\calS$.

\vspace{-2pt}
\subsection{Void matrices}\vspace{-2pt}
 A {\em void matrix\/} is a matrix with zero rows and/or zero columns. We will use the notation $0_{n,0}$ and $0_{0,m}$ to denote, respectively, $n\times 0$ and $0\times m$ void matrices. All matrix operations extend to void matrices in a natural manner. In particular, if $M$ and $N$ are, respectively $p\times q$ and $q\times r$ matrices, $MN$ is a $p\times r$ void matrix if $p=0$ or $r=0$ and $MN=0_{p,r}$ if $p,r\geq 1$ and $q=0$. In addition, the rank of a void matrix is zero.
\vspace{-2pt}
\subsection{Integer intervals and Hankel matrices}\vspace{-2pt}
Given $i,j\in\Z$ with $i\leq j$, we write $[i,j]$ to denote the ordered set of all integers between $i$ and $j$, including both $i$ and $j$. By convention, $[i,j]=\emptyset$ if $i>j$. 

Let $i,j\in\Z$ with $i\leq j$. Also, let $f(i),f(i+1),\dots,f(j) \in \R^n$ be a sequence of vectors. We define
$$
f_{[i,j]}:=\bbm f(i)&f(i+1)&\cdots& f(j)\ebm\in\R^{n\times (j-i+1)}.
$$
Also, for $0 \leq k \leq j-i$ the \emph{Hankel matrix of $k+1$ block rows associated with $f_{[i,j]}$} is defined by:
\begin{equation*}
H_k(f_{[i,j]})=\bbm f(i)&\cdots&f(j-k)\\
\vdots&&\vdots\\
f(i+k)&\cdots&f(j)\ebm.
\end{equation*}
We say that $f_{[i,j]}$ is \emph{persistently exciting of order $k+1$} if $H_k(f_{[i,j]})$ has full row rank. We stress that $H_k$ has $k+1$ block rows, which is slightly different from the notation in most of the literature, but adopted here to be consistent with the paper \cite{CamlibelRapisarda2024}.  

\subsection{Input-state-output systems}
\label{s:isosystems}
Throughout the paper, we work with linear discrete-time input-state-output systems of the form
\bse\label{eq: generic linear system}
\begin{align}
\bm x(t+1)&=A \bm x(t)+B\bm u(t)\\
\bm y(t)&=C \bm x(t)+D \bm u(t)
\end{align}
\ese
where $n\geq 0$, $m,p\geq 1$, $A\in\Rnn$, $B\in\Rnm$, $C\in\Rpn$, and $D\in\Rpm$. 

For $k\geq -1$, we define the \emph{$k$-th observability matrix} by
\begin{equation}
\label{obsmatrix}
\Omega_k:=\begin{cases}
0_{0,n}&\mbox{ if } k=-1\\
\bbm \Omega_{k-1}\\CA^{k}\ebm&\mbox{ if } k\geq 0,
\end{cases} 
\end{equation}
the \emph{$k$-th controllability matrix} by
\begin{equation}
\label{contmatrix}
\Gamma_k:=\begin{cases}
0_{n,0}&\mbox{ if } k=-1\\
\bbm A^{k}B&\Gamma_{k-1}\ebm&\mbox{ if } k\geq 0 
\end{cases},
\end{equation}
and the \emph{$k$-th Toeplitz matrix of Markov parameters} by
\begin{equation}
\label{toepmatrix}
\Theta_k :=\begin{cases}
0_{0,0}&\mbox{ if } k=-1\\
\bbm \Theta_{k-1} & 0\\C\Gamma_{k-1}&D\ebm&\mbox{ if } k\geq 0. 
\end{cases}.
\end{equation} 
We denote the smallest integer $k\geq 0$ such that $\rank \Omega_k=\rank \Omega_{k-1}$ by $\ell(C,A)$. Note that $0\leq \ell(C,A)\leq n$ and if $n=0$ then $\ell(C,A)=0$. Moreover,  if $(C,A)$ is observable, then $\ell(C,A)$ is the observability index of the pair.  In the following we call $\ell(C,A)$ the \emph{lag} of the system; for a justification of this terminology, see statements (iv) and (vii) of \citep[Thm. 6]{Willems86a}. 

\subsection{Systems with $m$ inputs and $p$ outputs}
We identify the system \eqref{eq: generic linear system} with the matrix $\sa$. Given $m\geq 1$ and $p\geq 1$, we define the set of all systems with lag $\ell$ and $n$ states by
\[
\calS(\ell,n):=\{\sa\in\R^{(n+p)\times(n+m)} \mid \ell(C,A)=\ell\}.
\]
We also define
\begin{align*}
\calS(n)&:=\{\sa\in\calS(\ell,n) \mid \ell\in[0,n]\}\\
\calS&:=\{\sa\in\calS(n) \mid n\geq 0\}\\
\calM&:=\{\sa\in\calS \mid (A,B)\text{ is controllable and } (C,A) \text{ is observable}\}.
\end{align*}

\subsection{Isomorphic systems} Two systems $\sai\in\calS(n)$ with $i\in[1,2]$ 
are said to be {\em isomorphic\/} if $D_1=D_2$ and there exists a nonsingular matrix $S\in\Rnn$ such that $A_1=S\inv A_2 S$, $B_1=S\inv B_2$, and $C_1=C_2 S$. By extending this notion to a set of systems, we say that $\calS'\subseteq\calS(n)$ has the {\em isomorphism property\/} if any pair of systems in $\calS'$ is isomorphic. By convention, the empty set has the isomorphism property. 

\section{Recap of the fundamental lemma and the high-level problem}\label{sec:fl}

In this section we recap the fundamental lemma \cite{Willems05} and we sketch the problem of this paper at a high level. Let $m \geq 1$ and $p \geq 1$. Consider the linear discrete-time input-state-output system
\begin{equation}
\begin{aligned}
\label{e:true sys} 
\bm x(t+1)&=\ta \bm x(t)+\tb \bm u(t)\\
\bm y(t)&=\tc \bm x(t)+\td \bm u(t),
\end{aligned}
\end{equation}
where $\nt\geq 0$, $\ta \in\R^{\nt\times \nt}$, $\tb\in\R^{\nt\times m}$, 
$\tc\in\R^{p\times \nt}$, and $\td\in\Rpm$. We refer to the system \eqref{e:true sys} as the {\em true system}. Its lag will be denoted by $\lt:=\ell(\tc,\ta)$.

Throughout the paper, we assume that the true system is \emph{minimal}, i.e. both observable and controllable. We also assume to know upper bounds on the true lag and true state-space dimension. In other words, we have $L$ and $N$ such that:
\beq\label{e:bounds on lt and nt}
\lt\leq L\qand \nt\leq N.
\eeq 
Let $t\geq 1$. Consider the input-output data $(\dut{t-1},\dyt{t-1})$ generated by system \eqref{e:true sys} during the time window $[0,t-1]$. In what follows, we state a reformulation of the fundamental lemma \cite{Willems05}.

\begin{proposition}
\label{p:fl}
Consider the minimal system \eqref{e:true sys}. Suppose that $\lt \geq 1$, and that $N$ and $L$ satisfy \eqref{e:bounds on lt and nt}. Assume that $u_{[0,t-1]}$ is persistently exciting of order $N+L+1$. Then 
\begin{equation}
\label{rankcondFL}
\rank H_{L}\left( \begin{bmatrix}
u_{[0,t-1]} \\ y_{[0,t-1]}
\end{bmatrix}\right) = (L+1)m+\nt.
\end{equation} 
\end{proposition}

The latter rank condition implies that all length $L+1$ input-output trajectories of \eqref{e:true sys} can be obtained as linear combinations of the columns of the data Hankel matrix in \eqref{rankcondFL}. This fact has been used frequently in the recent data-driven control literature \cite{Markovsky2008,Coulson19,DePerzik20,Berberich21}. In terms of system identification, Proposition~\ref{p:fl} has the following two consequences:
\ben[label=\normalfont{(}\roman*\normalfont{)}, ref=(\roman*)]
	\item The state-space dimension of the true system can be obtained from data as
	$$
	\nt = \rank H_{L}\left( \begin{bmatrix}
u_{[0,t-1]} \\ y_{[0,t-1]}
\end{bmatrix}\right) - (L+1)m.
	$$
	\item The system matrices $\ta,\tb,\tc$ and $\td$ can be identified up to an isomorphism from the data. That is, using $(\dut{t-1},\dyt{t-1})$ we can find matrices $A,B,C$ and $D$ such that 
	$$\begin{bmatrix}
	A & B \\ C & D
\end{bmatrix} \:\:\: \text{ and } \:\:\:\begin{bmatrix}
	\ta & \tb \\ \tc & \td
\end{bmatrix}$$ 
are isomorphic, using, for example, subspace identification \cite{vanOverschee1996}. 
\een

Thus, the fundamental lemma can be interpreted as an experiment design result, that tells us that the system matrices can be uniquely identified (up to isomorphism) as long as the input is designed to be persistently exciting. This condition on the input, however, puts a bound on the required number of data samples. In fact, $u_{[0,t-1]}$ can only be persistently exciting of order $N+L+1$ if 
$$
t \geq N+L+m(N+L+1).
$$
In fact, by designing $u_{[0,t-1]}$ such that $H_{N+L}(u_{[0,t-1]})$ is square and nonsingular, the system matrices can be identified from an input-output experiment of length exactly $N+L+m(N+L+1)$. However, as we will see in this paper it is not necessary to collect this many samples in order to identify the system. The main questions of this paper are thus the following: 
\begin{enumerate}
\item What is the shortest experiment for linear system identification? In other words, what is the smallest time $T$ for which there exists an ``informative" experiment $(u_{[0,T-1]},y_{[0,T-1]})$ from which the matrices $\ta,\tb,\tc$ and $\td$ can be identified up to isomorphism?
\item If possible, how can we design precisely $T$ inputs such that the resulting data $(u_{[0,T-1]},y_{[0,T-1]})$ are informative for system identification? 
\end{enumerate}

In order to give a precise answer to these questions, we first recap the definition and characterization of informativity for system identification \cite{CamlibelRapisarda2024} in Section~\ref{sec:inf}. We then provide a formal mathematical problem statement in Section~\ref{sec:problemstatement}, and solve this problem in Section~\ref{sec:experimentdesign}.

\section{Informativity for system identification}\label{sec:inf}
In this section we recap the definition and characterization of informativity for system identification from \cite{CamlibelRapisarda2024}. Without making any a priori assumption on the input, let $(u_{[0,t-1]},y_{[0,t-1]})$ be data obtained from \eqref{e:true sys}. This means that there exists $\dxt{t}\in\R^{\nt\times(t+1)}$ such that
\begin{equation}
\label{stateseqtruesys}
\bbm\dxtp{t}\\\dyt{t-1}\ebm
=\bbm \ta& \tb\\\tc& \td\ebm
\bbm x_{[0,t-1]} \\\dut{t-1}\ebm.
\end{equation}
A system $\sa\in\calS(n)$ {\em explains\/} the data $(\dut{t-1},\dyt{t-1})$ if there exists $\dxt{t}\in\R^{n\times(t+1)}$ such that
$$
\bbm\dxtp{t}\\\dyt{t-1}\ebm
=\bbm A& B\\C& D\ebm
\bbm x_{[0,t-1]} \\\dut{t-1}\ebm.
$$
The set of all systems that explain the data $(\dut{t-1},\dyt{t-1})$ is denoted by $\calE_t$ and is referred to as the set of {\em explaining systems}. The subsets of $\calE_t$ consisting of  systems with a given lag and state space dimension are respectively defined as
\[
\sylt{\ell}{n}:=\calE_t\cap\calS(\ell,n)\qand
\syt{n}:=\calE_t\cap\calS(n).
\]

\subsection{Definition of informativity for system identification}
The set $\calS_{L,N}$ consists of all systems with lag at most $L$ and state-space dimension at most $N$, i.e.,
\[
\calS_{L,N}:=
\bigcup_{\scriptsize
\bmat
\ell\in[0,L]\\
n\in[0,N]
\emat
}
\calS(\ell,n).
\]
In view of the bounds \eqref{e:bounds on lt and nt} and the minimality of the true system, we have the following \emph{prior knowledge}:
$$
\bbm \ta& \tb\\\tc& \td\ebm \in \calM \cap \calS_{L,N}.
$$

With this in mind, we define the notion of informativity for system identification \cite{CamlibelRapisarda2024}. We also refer the reader to \cite{vanWaarde2020} for a general treatment of informativity for different analysis and control problems.  
\begin{definition}
\label{def:IWPKC}
We say that the data $(\dut{t-1},\dyt{t-1})$ {\em are informative for system identification\/} if
\ben[label=\normalfont{(}\roman*\normalfont{)}, ref=(\roman*)]
\item\label{i:first} $\calE_t\cap \calM \cap\calS_{L,N}=\syt{\nt}\cap \calM \cap\calS_{L,N}$, and
\item\label{i:second} $\calE_t\cap \calM \cap\calS_{L,N}$ has the isomorphism property.
\een
\end{definition}
\noi The first condition means that all explaining systems satisfying the prior knowledge have $\nt$ states, while the second one asserts that any pair of such systems is isomorphic. Definition~\ref{def:IWPKC} thus captures the important property that there is precisely one equivalence class of state-space systems explaining the input-output data. In \citep{CamlibelRapisarda2024}, necessary and sufficient conditions are provided under which the data are informative for system identification. In order to recall these conditions, we first need to define two important integers, namely the shortest lag and minimum number of states.

\subsection{The shortest lag and minimum number of states}
Given the data $(u_{[0,t-1]},y_{[0,t-1]})$, we define the following two integers that play a pivotal role in the characterization of informativity for system identification:
\begin{align*}
\lmt&:=\min\{\ell\geq0 \mid \sylt{\ell}{n}\neq \emptyset \text{ for some } n\geq0 \}\\
\nmt&:=\min\{n\geq 0 \mid \syt{n}\neq \emptyset \}.
\end{align*}

As proven in \cite{CamlibelRapisarda2024}, these integers admit a simple characterization in terms of the data. To explain this, let $k\in[0,t-1]$ and denote the Hankel matrix of $k+1$ block rows constructed from the data $(\dut{t-1},\dyt{t-1})$ by 
\[
H_{k,t}:=\bbm
H_k(\dyt{t-1})\\
H_k(\dut{t-1})
\ebm
=
\begin{bmatrix} 
y(0)&\cdots&y(t-k-1)\\
y(1)&\cdots&y(t-k)\\
\vdots&&\vdots\\
y(k)&\cdots&y(t-1)\\
u(0)&\cdots&u(t-k-1)\\
u(1)&\cdots&u(t-k)\\
\vdots&&\vdots\\
u(k)&\cdots&u(t-1)
\end{bmatrix}.
\]
We also define
\[
G_{k,t}:=\bbm
H_{k-1}(\dyt{t-2})\\
H_k(\dut{t-1})
\ebm
.
\]
Note that $G_{k,t}$ may be obtained by removing the last row of outputs $y_{[k,t-1]}$ from $H_{k,t}$.
Now, define 
$$
\delta_{k,t}:=
\begin{cases}
p&\text{if } k=-1\\
\rank H_{k,t}-\rank G_{k,t}&\text{if } k\in[0,t-1] .
\end{cases}
$$
Note that 
$$
p\geq \delta_{k,t}\geq 0 \text{ for all } k\in[-1,t-1].
$$
Throughout the paper, we assume that $\dut{t-1}\neq 0_{m,t}$. From this blanket assumption, it follows that $\rank H_{t-1,t}=\rank G_{t-1,t}=1$ and hence 
\beq\label{e: rho T-1}
\delta_{t-1,t}=0.
\eeq

Let $q_t\in[0,t-1]$ be the smallest integer such that $\delta_{q_t,t}=0$. Note that $q_t$ is well-defined due to \eqref{e: rho T-1}. The shortest lag and minimum number of states can be computed in terms of $\delta_{k,t}$ and $q_t$, as recalled next (see \cite[Thm. 8]{CamlibelRapisarda2024}). 

\begin{proposition}\label{p:nmin} $\lmt=q_t$ and $\nmt=\sum_{i=0}^{\lmt}\delta_{i,t}$. 
\end{proposition}

An important consequence of Proposition~\ref{p:nmin} is that the integers $\lmt$ and $\nmt$ can readily be computed using the data.

\subsection{Necessary and sufficient conditions for informativity}
We are now in a position to recall the conditions for informativity for system identification. Before we do so, we note that \cite[Thms. 6(b) and 8]{CamlibelRapisarda2024} shows that $N-\nmt+\lmt$ is an upper bound for the lag of {\em any\/} explaining system with at most $N$ states. This upper bound, which is determined by the data and $N$, is in some cases smaller than the given upper bound $L$. This means that we can replace $L$ by the {\em actual\/} upper bound on the lag:
\[
L_t^a:=\min(L,N-\nmt+\lmt).
\]  

The following theorem from \citep[Thm. 9]{CamlibelRapisarda2024} provides \emph{necessary and sufficient} conditions for the data to be informative for system identification. 

\bthe\label{t:main}
The data $(u_{[0,t-1]},y_{[0,t-1]})$ are informative for system identification if and only if the following two conditions hold:
\bse\label{e:nec suff conditions}
\begin{gather}
t \geq L_t^a+(L_t^a+1)m+\nmt\label{e:lb T}\\
\rank H_{L_t^a,t}=(L_t^a+1)m+\nmt. \label{e:main rank condition}
\end{gather}
\ese
Moreover, if the conditions in \eqref{e:nec suff conditions} are satisfied, then 
\bse\label{e:extras}
\begin{gather}
\lt=\lmt\label{e:extra-lm}\\
\nt=\nmt\label{e:extra-nm}\\
\calE_t\cap\calM \cap \calS_{L,N}=\calE_t(\nmt).\label{e:extra-main}
\end{gather}
\ese
\ethe 

\section{Formal problem statement}
\label{sec:problemstatement}

If the data $(u_{[0,t-1]},y_{[0,t-1]})$ are informative for system identification, then $\lmt = \lt$ and $\nmt = \nt$ by Theorem~\ref{t:main}. In this case, $L_t^a$ is equal to
$$
L^a := \min(L,N-\nt+\lt).
$$
It follows from the lower bound \eqref{e:lb T} that 
\begin{equation} 
\label{defT}
t \geq T := L^a + (L^a+1)m + \nt,
\end{equation}
that is, any set of informative input-output data contains at least $T$ samples. 

The main question is now as follows: can we design a sequence of inputs $u_{[0,T-1]}$ of length \emph{precisely} $T$ such that the resulting input-output data $(u_{[0,T-1]},y_{[0,T-1]})$ are informative for system identification? We will focus on an \emph{online} design of the inputs, in the sense that the choice of $u(t)$ is guided by the data $(u_{[t-1]},y_{[0,t-1]})$ collected at previous time steps. We formalize the problem as follows. 

\begin{problem}
\label{prob:problemformulation}
Let $T$ be as in \eqref{defT}. Consider the system \eqref{e:true sys} with initial state $x(0) = x_0 \in \mathbb{R}^n$. Let $y(t) \in \mathbb{R}^p$ denote the output of \eqref{e:true sys} at time $t$ resulting from $x_0$ and the control inputs $u(0),u(1),\ldots,u(t) \in \mathbb{R}^m$. 

For each $t = 0,1,\ldots,T-1$, given $(u_{[0,t-1]},y_{[0,t-1]})$, design $u(t)$ such that the resulting data $(u_{[0,T-1]},y_{[0,T-1]})$ are informative for system identification. 
\end{problem}

We note that the initial state $x_0$ of the system \eqref{e:true sys} is arbitrary and not assumed to be given. Our goal is thus to design inputs that lead to an informative experiment \emph{irrespective of $x_0$}. 

In addition, we emphasize that it is not straightforward to see that Problem~\ref{prob:problemformulation} has a solution. In fact, even though $T$ is a lower bound on the number of data samples required for system identification, it is at this point unclear whether there \emph{exists} an experiment of length exactly $T$. Also, even if such an experiment exists, it is far from obvious that there is a systematic way of \emph{constructing} such an experiment without knowledge of the true system. An additional challenge is that the time $T$ itself depends on the true lag and true state-space dimension, which are \emph{not a priori known}. 

Remarkably, as we show in this paper, it turns out to be \emph{always} possible to design an informative experiment of length precisely $T$, despite these challenges.

\section{Online experiment design}
\label{sec:experimentdesign}
In this section we present our main results, building up to the online experiment design method. We start with the following auxiliary lemma that asserts that the rank of the Hankel matrix $H_{k,t}$ can be increased at time $t+1$, assuming that certain conditions are met.  

\begin{lemma} 
\label{l:increaserank}
Let $t \geq 2$ and $k \geq 1$. If 
\begin{equation}
\label{conditionGH}
\rank G_{k,t} < m+\rank H_{k-1,t}
\end{equation} 
then there exists an $(m-1)$-dimensional affine set $\calA_t \subseteq\R^m$ such that
\begin{equation}
\label{rankincrease}
\rank H_{k,t+1} =\rank H_{k,t}+1 
\end{equation} 
whenever $u(t) \not\in \calA_t$.
\end{lemma} 

\begin{proof}
Note that $\lk H_{k-1,t} \times \{0_m\} \subseteq \lk G_{k,t}$. Therefore, it holds that $\dim \lk G_{k,t} \geq \dim \lk H_{k-1,t}$. It follows from the rank-nullity theorem that $kp +(k+1)m - \rank G_{k,t} \geq k(p+m) - \rank H_{k-1,t}$. As such, $\rank G_{k,t} \leq m + \rank H_{k-1,t}$. Moreover, note that $\rank G_{k,t} = m + \rank H_{k-1,t}$ if and only if $\lk H_{k-1,t} \times \{0_m\} = \lk G_{k,t}$. Therefore, \eqref{conditionGH} implies that there exist $\xi_i \in \mathbb{R}^p$ and $\eta_j \in \mathbb{R}^m$ with $i \in [0,k-1]$ and $j \in [0,k]$ such that $\eta_k \neq 0$ and 
$$
\begin{bmatrix}
\xi_0^\top & \xi_1^\top & \cdots & \xi_{k-1}^\top & \eta_0^\top & \eta_1^\top & \cdots & \eta_k^\top
\end{bmatrix} G_{k,t} = 0. 
$$
Now, define the set
$$
\mathcal{A}_t := \{v \in \mathbb{R}^m \mid \eta_k v + \sum_{i \in [0,k-1]} \xi_i^\top y(t-k+i) + \eta_i^\top u(t-k+i) = 0 \}.
$$
If $u(t) \not\in \mathcal{A}_t$ then 
$$
\begin{bmatrix}
\xi_0^\top & \xi_1^\top & \cdots & \xi_{k-1}^\top & \eta_0^\top & \eta_1^\top & \cdots & \eta_k^\top
\end{bmatrix} G_{k,t+1} \neq 0. 
$$
Since $\lk G_{k,t+1} \subseteq \lk G_{k,t}$, we conclude from the latter inequality that $\dim \lk G_{k,t+1} < \dim \lk G_{k,t}$. Therefore, the last column of $G_{k,t+1}$ is not a linear combination of the columns of $G_{k,t}$. Thus, the last column of $H_{k,t+1}$ is also not a linear combination of the columns of $H_{k,t}$. We conclude that \eqref{rankincrease} holds, which proves the lemma. 
\end{proof}

As long as the inequality \eqref{conditionGH} holds, Lemma~\ref{l:increaserank} may be successively applied several times to increase the rank of the Hankel matrix. In the next lemma, we show how to deduce from the data whether $k = L^a$, as soon as the condition \eqref{conditionGH} fails to hold. 
This lemma will be used as a stopping criterion for our online experiment design algorithm.

\begin{lemma}
\label{l:stoppingcriterion}
Let $t \geq 2$ and $k \geq 1$ and suppose that the data $(u_{[0,t-1]},y_{[0,t-1]})$ are such that $H_{k,t}$ has full column rank. Let $\tau \geq 0$ and assume that $u_{[t,t+\tau -1]}$ satisfy
\begin{enumerate}
\item \label{stopcond1}
 for every $s \in [t,t+\tau-1]$, $\rank G_{k,s} < m + \rank H_{k-1,s}$ and $u(s) \not\in \calA_s$, where $\calA_s$ is as in Lemma~\ref{l:increaserank}. 
\item \label{stopcond2}
 $\rank G_{k,t+\tau} = m + \rank H_{k-1,t+\tau}$. 
\end{enumerate}
Then the following statements hold:
\begin{enumerate}[label = (\alph*)]
\item If $k \geq \lt$ then
\label{lemmastopa}
\begin{itemize}
\item $\rank H_{k,t+\tau} = (k+1)m + \nt$,
\item $t+\tau = k + (k+1)m + \nt$,
\item $\ell_{\mathrm{min},t+\tau} = \lt$, and $n_{\mathrm{min},t+\tau} = \nt$.
\end{itemize}
\item If $k = L^a_{t+\tau}$ then
\label{lemmastopb}
\begin{itemize}
\item $k = L^a$,
\item $t+\tau = T$, and
\item $(u_{[0,T-1]},y_{[0,T-1]})$ are informative for system identification. 
\end{itemize}
\end{enumerate}
\end{lemma} 

\begin{proof}
We first prove \ref{lemmastopa}. Assume that $k \geq \lt$. Hypothesis~\ref{stopcond1} and Lemma~\ref{l:increaserank} imply that 
\begin{equation}
\label{rankHkttau}
\rank H_{k,t+\tau} = \rank H_{k,t} + \tau = t +\tau-k.
\end{equation}
Let $x_{[0,t+\tau-1]} \in \mathbb{R}^{\nt \times (t+\tau)}$ be a state compatible with the input-output data $(u_{[0,t+\tau-1]},y_{[0,t+\tau-1]})$ and the true system, i.e., \eqref{stateseqtruesys} holds. Since $(\tc,\ta)$ is observable and $k \geq \lt$, the observability matrix $\Omega_{k-1}$ of the true system has rank $\nt$. This implies that the matrices
\beq\label{e:Phi and Psi}
\Phi_{k-1} := \bbm
\Omega_{k-1}&\Theta_{k-1} \\
0&I
\ebm\text{ and }
\Psi_k :=
\bbm
\Omega_{k-1}&\Theta_{k-1} &0 \\
0&I&0\\
0&0&I_m
\ebm
\eeq
have full column rank, where we recall that $\Theta_{k-1}$ is the Toeplitz matrix of Markov parameters of the true system, defined in \eqref{toepmatrix}. Therefore,
$$
\rank H_{k-1,t+\tau} = \rank \left( \Phi_{k-1}
\begin{bmatrix}
x_{[0,t+\tau-k]} \\ H_{k-1}(u_{[0,t+\tau-1]})
\end{bmatrix} \right) = \rank \begin{bmatrix}
x_{[0,t+\tau-k]} \\ H_{k-1}(u_{[0,t+\tau-1]})
\end{bmatrix}.
$$
Moreover,
\begin{align}
\label{rkGktau}
\rank G_{k,t+\tau} &= \rank \left( \Psi_k \begin{bmatrix}
x_{[0,t+\tau-k-1]} \\ H_{k}(u_{[0,t+\tau-1]}) \end{bmatrix} \right) = \rank \begin{bmatrix}
x_{[0,t+\tau-k-1]} \\ H_{k}(u_{[0,t+\tau-1]})
\end{bmatrix} \\
\label{rkHktau}
&= \rank H_{k,t+\tau}.
\end{align}
By hypothesis \ref{stopcond2}, we thus have
$$
\rank \begin{bmatrix}
x_{[0,t+\tau-k-1]} \\ H_{k}(u_{[0,t+\tau-1]})
\end{bmatrix} = m + \rank \begin{bmatrix}
x_{[0,t+\tau-k]} \\ H_{k-1}(u_{[0,t+\tau-1]})
\end{bmatrix}.
$$
By \cite[Lemma 23]{CamlibelRapisarda2024} and the fact that $(\ta,\tb)$ is controllable, 
$$
\rank \begin{bmatrix}
x_{[0,t+\tau-k-1]} \\ H_{k}(u_{[0,t+\tau-1]})
\end{bmatrix} = \nt + (k+1)m.
$$
Therefore, by \eqref{rkHktau}, we conclude that 
\begin{equation} 
\label{formularankHk}
\rank H_{k,t+\tau} = \nt + (k+1)m,
\end{equation}
proving the first item of \ref{lemmastopa}. The second item of \ref{lemmastopa} now follows immediately from \eqref{rankHkttau}. Finally, to prove the third item, let 
$$
\begin{bmatrix}
A & B \\ C & D
\end{bmatrix} \in \mathcal{E}(\ell_{\mathrm{min},t+\tau},n_{\mathrm{min},t+\tau}).
$$
Let $x_{[0,t+\tau-1]} \in \mathbb{R}^{n_{\mathrm{min},t+\tau} \times (t+\tau)}$ be a state compatible with the input-output data and the above explaining system. Obviously, $\rank H_{k,t+\tau} \leq n_{\mathrm{min},t+\tau} + (k+1)m$ and therefore $n_{\mathrm{min},t+\tau} \geq \nt$ by \eqref{formularankHk}. However, since also $n_{\mathrm{min},t+\tau} \leq \nt$, we obtain $\nt = n_{\mathrm{min},t+\tau}$. Finally, it follows from \cite[Equation (17)]{CamlibelRapisarda2024} that $\ell_{\mathrm{min},t+\tau} \geq \lt$. Since obviously $\ell_{\mathrm{min},t+\tau} \leq \lt$, we conclude that $\ell_{\mathrm{min},t+\tau} = \lt$, proving the third item of \ref{lemmastopa}.

Next, we will prove \ref{lemmastopb}. Assume that $k = L_{t+\tau}^a$. We have that 
$$
N - n_{\mathrm{min},t+\tau} + \ell_{\mathrm{min},t+\tau} \geq \nt - n_{\mathrm{min},t+\tau} + \ell_{\mathrm{min},t+\tau} \geq \lt,
$$
where the last inequality follows from \cite[Equation (17)]{CamlibelRapisarda2024}. Combining this with $L \geq \lt$, we obtain $k = L_{t+\tau}^a \geq \lt$, by definition of $L_{t+\tau}^a$. Therefore, the three items listed under \ref{lemmastopa} hold. From the fact that $\ell_{\mathrm{min},t+\tau} = \lt$ and $n_{\mathrm{min},t+\tau} = \nt$, it follows that $L_{t+\tau}^a = L^a$, and therefore $k = L^a$. This proves the first item of \ref{lemmastopb}. Moreover, the second item of \ref{lemmastopa} implies that $t + \tau = L^a + (L^a + 1)m + \nt$, which shows that $t + \tau = T$, proving the second item of \ref{lemmastopb}. Finally, from the third item of \ref{lemmastopa}, we see that $\rank H_{L_T^a,T} = (L^a_T+1)m + n_{\mathrm{min},T}$. Therefore, it follows from Theorem~\ref{t:main} that the data $(u_{[0,T-1]},y_{[0,T-1]})$ are informative for system identification. Therefore, also the third item of \ref{lemmastopb} holds. This proves the lemma. 
\end{proof}

The core idea of our approach is to adapt the depth $k$ of the Hankel matrix during the operation of the experiment design procedure. For a \emph{fixed} depth $k$, Lemma~\ref{l:increaserank} will be used for $s \in [t,t+\tau-1]$ until the rank condition $\rank G_{k,t+\tau} = m + \rank H_{k-1,t+\tau}$ holds. Then, following Lemma~\ref{l:stoppingcriterion}, we check whether $k = L_{t+\tau}^a$. If $k = L_{t+\tau}^a$ then we are done because the data $(u_{[0,t+\tau-1]},y_{[0,t+\tau-1]})$ are informative for system identification. Otherwise, if $k \neq L_{t+\tau}^a$ we increase the depth of the Hankel matrix to $k+1$ and repeat the process. This leads to the following algorithm.

\begin{algorithmic}[1]
\Procedure{OnlineExperiment}{$L,N$}
\State {\bf choose} $\du{0}{m-1}$ to be nonsingular \label{alg:nonsingular}
\State {\bf measure} outputs $y_{[0,m-1]}$
\State $t \gets m, k \gets 0$
\While {\text{$k \neq L_t^a$}}\Comment{stopping criterion} \label{alg:outerwhile}
\State $k \gets k+1$ \label{alg:increasek}
\If{\text{$ t = k$}} \label{alg:if}
\State {\bf choose} $u(t)$ arbitrarily \label{alg:inputsarbitrary}
\State {\bf measure} output $y(t)$ \label{alg:outputstoarbinputs}
\State {$t \gets t+1$}
\EndIf \label{alg:endif}
\While {\text{$ \rank G_{k,t} < m+\rank H_{k-1,t}$}} \label{alg:whileGH}
\State {\bf choose} $u(t) \not\in \calA_t$ \label{alg:unotinAt} 
\State {\bf measure} output $y(t)$
\Comment{$\rank H_{k,t+1} = \rank H_{k,t}+1$}
\State $t\gets t+1$
\EndWhile \label{alg:endwhile}
\EndWhile \label{alg:endouterwhile}
\State \textbf{return} $(\du{0}{t-1},\dy{0}{t-1})$ \\
\Comment{$ ({k},{t})=(L^a,T) $ and the data are informative}
\EndProcedure
\end{algorithmic}

\noindent 
The following theorem asserts that \textsc{OnlineExperiment}$(L,N)$ leads to informative data sets with the least possible number of samples. This is the main result of the paper.
\begin{theorem}
\label{t:mainresult}
The procedure \textsc{OnlineExperiment}$(L,N)$ returns input-output data $(\du{0}{t-1},\dy{0}{t-1})$ that are informative for system identification. Moreover, $t = T$, where $T$ is defined in \eqref{defT}.
\end{theorem}
Before we prove Theorem~\ref{t:mainresult}, we state the following auxiliary lemma. 

\begin{lemma}
\label{l:increasedepth}
Let $k \geq 0$ and $t \geq k+2$. If the data $(u_{[0,t-1]},y_{[0,t-1]})$ are such that $H_{k,t}$ has full column rank then also $H_{k+1,t}$ has full column rank. 
\end{lemma}
\begin{proof}
Since $H_{k,t}$ has full column rank, the submatrix
\begin{equation}
\label{submatrixHkt}
\begin{bmatrix} 
y(1)&\cdots&y(t-k-1)\\
\vdots&&\vdots\\
y(k+1)&\cdots&y(t-1)\\
u(1)&\cdots&u(t-k-1)\\
\vdots&&\vdots\\
u(k+1)&\cdots&u(t-1)
\end{bmatrix},
\end{equation}
obtained from removing the first column of $H_{k,t}$, has full column rank as well. Since \eqref{submatrixHkt} is also the submatrix of $H_{k+1,t}$ obtained by removing the row blocks $u_{[0,t-k-2]}$ and $y_{[0,t-k-2]}$, we conclude that $H_{k+1,t}$ has full column rank. This proves the lemma. 
\end{proof}

\begin{proof}[Proof of Theorem~\ref{t:mainresult}]
The proof consists of the following three steps. First, we prove that the Hankel matrix $H_{k,t}$ always has full column rank at the start of the while loop in Line~\ref{alg:whileGH}. Secondly, we prove that the procedure terminates within a finite number of steps. Finally, we show that the latter number is precisely equal to $T$, and the data $(u_{[0,T-1]},y_{[0,T-1]})$  are informative for system identification. 

We begin with the first step. Consider the $k$-th iteration of the while loop in Lines~\ref{alg:outerwhile}-\ref{alg:endouterwhile}. Let $t_k$ be the time instant at the start of the while loop in Line~\ref{alg:whileGH}. We claim that $H_{k,t_k}$ has full column rank. 

We first prove this claim for the first iteration of the while loop, i.e., consider $k=1$. If $m \geq 2$ then the if statement in Lines~\ref{alg:if}-\ref{alg:endif} is ignored, and $t_1 = m$. Note that the Hankel matrix $H_{0,m}$ has full column rank by the choice of the inputs $u_{[0,m-1]}$ in Line~\ref{alg:nonsingular}. It is then clear that $H_{1,t_1}$ has full column rank by Lemma~\ref{l:increasedepth}. On the other hand, if $m = 1$, then $t_1 = m+1$. In this case, the Hankel matrix $H_{1,m+1}$ has one column, which is nonzero due to line~\ref{alg:nonsingular}. Therefore, also in this case $H_{1,t_1}$ has full column rank.

Now consider any iteration $k \geq 1$ of the while loop in Lines~\ref{alg:outerwhile}-\ref{alg:endouterwhile}. Assume that $H_{k,t_k}$ has full column rank. Our goal is to show that $H_{k+1,t_{k+1}}$ has full column rank as well.

Consider the $k$-th iteration of the while loop in Lines~\ref{alg:outerwhile}-\ref{alg:endouterwhile}. Since $H_{k,t_k}$ has full column rank by hypothesis, we can apply Lemma~\ref{l:increaserank} to the while loop in Lines~\ref{alg:whileGH}-\ref{alg:endwhile}. In particular, this while loop is applied for a finite number of iterations, say $\tau \in \mathbb{N}$, which yields the Hankel matrix $H_{k,t_k+\tau}$. By repeated application of Lemma~\ref{l:increaserank}, $H_{k,t_k+\tau}$ has full column rank. This means, in particular, that $t_k + \tau \geq k+1$.

Now, we turn our attention to the $(k+1)$-th iteration of the while loop in Lines~\ref{alg:outerwhile}-\ref{alg:endouterwhile}. If $t_k+\tau = k+1$ then the if statement in Lines~\ref{alg:if}-\ref{alg:endif} generates an arbitrary input $u(t_k+\tau)$ and corresponding output $y(t_k+\tau)$. In this case, $t_{k+1} = t_k+\tau+1$ and the resulting Hankel matrix $H_{k+1,t_{k+1}}$ is a column vector of rank one by Line~\ref{alg:nonsingular}. In the other case, if $t_k + \tau > k+1$ then $t_{k+1} = t_k+\tau$ and it follows that $H_{k+1,t_{k+1}}$ has full column rank by Lemma~\ref{l:increasedepth}.

Secondly, we prove that the procedure terminates in a finite number of steps. Now, let $t_k \in \mathbb{N}$ be the time instant, corresponding to the depth $k$, at which the stopping criterion in Line~\ref{alg:outerwhile} is checked. We want to prove the existence of a depth $k \geq 0$ such that $k = L_{t_k}^a$. Clearly, since $L_{t_k}^a \geq \lt$, this cannot happen if $k < \lt$. For any $k \geq \lt$ we have that $\ell_{\mathrm{min},t_k} = \lt$ and $n_{\mathrm{min},t_k} = \nt$ by Lemma~\ref{l:stoppingcriterion}\ref{lemmastopa}, meaning that $L_{t_k}^a = L^a$. Since the depth $k$ is increased by one in every iteration of the while loop in Lines~\ref{alg:outerwhile}-\ref{alg:endouterwhile}, this implies that there exists $k \geq \lt$ such that $k = L_{t_k}^a$. 

Finally, we prove the last step. Let $k$ be such that $k = L_{t_k}^a$. It follows from Lemma~\ref{l:stoppingcriterion}\ref{lemmastopb} that $t_k$ is precisely equal to $T$, and the data $(u_{[0,T-1]},y_{[0,T-1]})$ are informative for system identification. This proves the theorem. 
\end{proof}

\begin{remark}
Without going into details, we mention that Theorem~\ref{t:mainresult} shows that ``randomly" chosen inputs $u_{[0,T-1]}$ lead to informative experiments of length precisely $T$ with high probability. In fact, the only imposed constraints on the inputs are that $u_{[0,m-1]}$ is nonsingular (Line~\ref{alg:nonsingular} of the algorithm), and that $u(t) \in \mathbb{R}^m$ is not a member of an $(m-1)$-dimensional affine set (Line~\ref{alg:unotinAt}). Regardless of how the inputs are chosen to satisfy these constraints, however, a crucial aspect of \textsc{OnlineExperiment}$(L,N)$ is its stopping criterion. Indeed, we emphasize that $T$ is not a priori known but has to be deduced from data. 
\end{remark} 

\section{Illustrative example}
\label{sec:example}
In this section we illustrate \textsc{OnlineExperiment}$(L,N)$ by means of an example. Consider the true system 
\[
\struebig=
\left[
\begin{array}{ccc|cc}
0 & 1 & 0 & 1 & 0\\
0 & 0 & 1 & 0 & 1\\
0 & 0 & 0 & 0 & 1\\
\hline
1 & 0 & 0 & 1 & 0\\
0 & 1 & 0 & 0 & 0
\end{array}
\right],
\]
so that $\nt=3$, $m=2$, $p=2$, and $\lt=2$. Let $x_0=\begin{bmatrix}
1 & 1 & 0
\end{bmatrix}^\top$. Moreover, let $L=4$ and $N=4$ be the available upper bounds.

\subsection{Online experiment}
We now apply \textsc{OnlineExperiment}$(L,N)$. We start with $k=0$. We choose $u_{[0,1]} = I$, which is obviously nonsingular, and measure
\[
y_{[0,1]}=
\begin{bmatrix}
2 & 2\\
1 & 0
\end{bmatrix}.
\]
Using these data and Proposition~\ref{p:nmin}, we compute
\[
\ell_{\mathrm{min},2}=0,\:\: n_{\mathrm{min},2}=0, \:\: \text{and }\: L_2^a=4.
\]
Since $k \neq L_2^a$, we increase $k$ to $k=1$. The inputs are now designed according to the while loop in Line~\ref{alg:whileGH} as:
\[
u_{[2,7]}=
\begin{bmatrix}
1  &   1  &   1  &   1  &   1  &   1\\
0  &   0  &   0  &   0  &   0  &   1
\end{bmatrix},
\]
resulting in the measured outputs
\[
y_{[2,7]}=
\begin{bmatrix}
1  &   3  &   3  &   2  &   2  &   2\\
1  &   1  &   0  &   0  &   0  &   0
\end{bmatrix}.
\]
It can be checked that we have increased the rank of the depth-$1$ Hankel matrix from $\rank H_{1,3}=2$ to $\rank H_{1,8}=7$. Based on the data thusfar, we use Proposition~\ref{p:nmin} to compute:
\[
\ell_{\mathrm{min},8}=2,\:\: n_{\mathrm{min},8}=3, \:\: \text{and }\: L_8^a=3.
\]
Since $k \neq L_8^a$, we set $k=2$. Following the while loop in Line~\ref{alg:whileGH}, we construct
\[
u_{[8,10]}=
\begin{bmatrix}
1  &   1  &   1\\
0  &   0  &   0
\end{bmatrix},
\]
which yields
\[
y_{[8,10]}=
\begin{bmatrix}
2  &  3  &   3\\
1  &  1  &   0
\end{bmatrix}.
\]
By doing so, we have increased the rank of the depth-$2$ Hankel matrix from $\rank H_{2,9}=7$ to $\rank H_{2,11}=9$. Again, we compute: 
\[
\ell_{\mathrm{min},11}=2,\:\: n_{\mathrm{min},11}=3, \:\: \text{and }\: L_{11}^a=3.
\]
Since $k \neq L_{11}^a$ we set $k=3$. This time, we apply the while loop in Line~\ref{alg:whileGH} to obtain the data
\[
u_{[11,13]}=
\begin{bmatrix}
1  &   0  &   0\\
0  &   1  &   1
\end{bmatrix}, \text{ and } 
y_{[11,13]}=
\begin{bmatrix}
2  &   1  &   0\\
0  &   0  &   1
\end{bmatrix}.
\]
The rank of the depth-$3$ Hankel matrix has increased from $\rank H_{3,12}=9$ to $\rank H_{3,14}=11$. Finally, we compute
\[
\ell_{\mathrm{min},14}=2,\:\: n_{\mathrm{min},14}=3, \:\: \text{and }\: L_{14}^a=3.
\]
Since $k = L_{14}^a$, the procedure terminates. We conclude that $T = 14$ and the data $(u_{[0,13]},y_{[0,13]})$ are informative for system identification. For this example, we note that the required number of samples $T$ is less than the experiment design approach in \cite{Waarde21} that works with the \emph{fixed} depth $L=4$ Hankel matrix. However, this is not always the case. For example, if we study the same example but with the given upper bounds $L = 3$ and $N = 6$, we can use \textsc{OnlineExperiment}$(L,N)$ to generate the same informative data, with the only difference that we now have $L_2^a = 3$. In this case, the number of $T =14$ data samples is the same as in \cite{Waarde21}.

\subsection{PE of order $L^a+1$ is not sufficient for informativity}
According to Theorem~\ref{t:main}, an obvious necessary condition for informativity is that the inputs are \emph{persistently exciting of order $L^a +1$}. This condition, however, is not sufficient as demonstrated next. We use the same example as above, but just change $u(12)$ from $\begin{bmatrix}
0\\1
\end{bmatrix}$ to $\begin{bmatrix}
1\\0
\end{bmatrix}$, i.e., we choose the inputs as
\[
u_{[0,13]}=
\begin{bmatrix}
1 &   0 &   1 &   1 &   1 &   1 &   1 &   1 &   1 &   1 &   1 &   1 &   1 &   0\\
0 &   1 &   0 &   0 &   0 &   0 &   0 &   1 &   0 &   0 &   0 &   0 &   0 &   1
\end{bmatrix}.
\]
The corresponding outputs are then given by
\[
y_{[0,13]}=
\begin{bmatrix}
2 &  2 &  1 &  3 &  3 &  2 &  2 &  2 &  2 &  3 &  3 &  2 &  2 &  1\\
1 &  0 &  1 &  1 &  0 &  0 &  0 &  0 &  1 &  1 &  0 &  0 &  0 &  0
\end{bmatrix}.
\]
In this case, $\rank H_3(u_{[0,13]})=8$ so $u_{[0,13]}$ is persistently exciting of order $L^a+1$. However $\rank H_{3,14}=10\neq 11$ so the conditions of Theorem~\ref{t:main} are not satisfied.

\section{Comparison to previous work}
\label{sec:comparison}

The proposed online experiment design method largely improves the (offline) persistency of excitation condition of the fundamental lemma \citep{Willems05} (see Proposition~\ref{p:fl}). Indeed, recall that the input $u_{[0,t-1]}$ can only be \emph{persistently exciting} of order $N+L+1$ if $t \geq N+L+m(N+L+1)$. In general, this lower bound on the number of data samples is much larger than $T$ in \eqref{defT}. For example, if $m = 80$, $p = 10$, $\lt = 20$, $\nt = 100$, $L = 100$ and $N = 150$, the online experiment design method requires $T = 5850$ samples whereas persistency of excitation requires $t \geq 20330$ samples. 

The proposed approach also improves the online experiment design of \citep{Waarde21}. In fact, in the latter paper a method was given to guarantee that the Hankel matrix $H_L$ of \emph{fixed depth} $L$ has rank $(L+1)m+\nt$. This was done in the least possible number of time steps, $t = L +(L+1)m +\nt$. However, by Theorem~\ref{t:main}, the condition $\rank H_L = (L+1)m+\nt$ is sufficient for informativity for system identification, but in general not necessary. In particular, if $L^a < L$ then the experiment design method of this paper leads to a shorter experiment for system identification than the one provided in \citep{Waarde21}. If $L^a = L$, then the number of samples coincides with \citep{Waarde21}.

\section{Conclusions}
\label{sec:conclusions}

In this paper we have proposed an experiment design method that leads to input-output data that are informative for system identification. The key features of the approach are i) it is \emph{online}, meaning that the design of the inputs is guided by data collected at previous time steps, and ii) it \emph{adapts the depth} of the input-output Hankel matrix during the operation of the algorithm. We have shown that this approach leads to informative sequences of input-output samples of the shortest possible length. Interestingly, the exact number of samples in such a shortest experiment cannot be determined a priori, but is only revealed after the termination of the procedure. The online experiment design method improves over methods based on persistent excitation \cite{Willems05} by significantly reducing the required number of data samples for system identification. The results of this paper have also revealed that the online experiment design method of \cite{Waarde21} yields the shortest experiment for system identification only in \emph{some} cases. In situations where the data-based bound on the lag of the system is smaller than the a priori given bound, the experiment design method of this paper outperforms \cite{Waarde21}.

\bibliographystyle{plain}
\bibliography{ifacconf}
\end{document}

%% file: ka-newcommands.tex
\DeclareMathOperator{\rank}{rank}

\let\leq\leqslant
\let\geq\geqslant
\let\emptyset\varnothing

\newcommand{\calA}{\ensuremath{\mathcal{A}}}

\newcommand{\calE}{\ensuremath{\mathcal{E}}}

\newcommand{\calM}{\ensuremath{\mathcal{M}}}

\newcommand{\calS}{\ensuremath{\mathcal{S}}}

\newcommand{\hatx}{\ensuremath{\hat{x}}}

\newcommand{\bmat}{\begin{matrix}}
\newcommand{\emat}{\end{matrix}}
\newcommand{\bbm}{\begin{bmatrix}}
\newcommand{\ebm}{\end{bmatrix}}
\newcommand{\bbma}{\begin{bmatrix*}[r]}
\newcommand{\ebma}{\end{bmatrix*}}
\newcommand{\bpm}{\begin{pmatrix}}
\newcommand{\epm}{\end{pmatrix}}
\newcommand{\bvm}{\begin{vmatrix}}
\newcommand{\evm}{\end{vmatrix}}
\newcommand{\bse}{\begin{subequations}}
\newcommand{\ese}{\end{subequations}}
\newcommand{\beq}{\begin{equation}}
\newcommand{\eeq}{\end{equation}}
\newcommand{\ben}{\renewcommand{\labelenumi}{\arabic{enumi}.}
\renewcommand{\theenumi}{\arabic{enumi}}\begin{enumerate}}

\newcommand{\een}{\end{enumerate}}

\newcommand{\bena}{\renewcommand{\labelenumi}{\alph{enumi}.}
\renewcommand{\theenumi}{\alph{enumi}}\begin{enumerate}}

\newcommand{\eena}{\end{enumerate}}

\newcommand{\bit}{\begin{itemize}}
\newcommand{\eit}{\end{itemize}}
\newcommand{\bthe}{\begin{theorem}}
\newcommand{\ethe}{\end{theorem}}
\newcommand{\blem}{\begin{lemma}}
\newcommand{\elem}{\end{lemma}}
\newcommand{\bprop}{\begin{proposition}}
\newcommand{\eprop}{\end{proposition}}
\newcommand{\bex}{\begin{example}}
\newcommand{\eex}{\end{example}}
\newcommand{\bas}{\begin{assumption}}
\newcommand{\eas}{\end{assumption}}
\newcommand{\bre}{\begin{remark}}
\newcommand{\ere}{\end{remark}}
\newcommand{\bcor}{\begin{corollary}}
\newcommand{\ecor}{\end{corollary}}
\newcommand{\bdfn}{\begin{definition}}
\newcommand{\edfn}{\end{definition}}
\newcommand{\bcon}{\begin{conjecture}}
\newcommand{\econ}{\end{conjecture}}

\newcommand{\inv}{\ensuremath{^{-1}}}

\newcommand{\R}{\ensuremath{\mathbb R}}
\newcommand{\C}{\ensuremath{\mathbb C}}

\newcommand{\Z}{\ensuremath{\mathbb Z}}

\newcommand{\qand}{\quad\text{ and }\quad}

%% file: ka-newcommands-extra.tex
\newcommand{\Rnn}{\R^{n\times n}}

\newcommand{\Rnm}{\R^{n\times m}}
\newcommand{\Rpn}{\R^{p\times n}}
\newcommand{\Rpm}{\R^{p\times m}}

\newcommand{\ta}{A_{\mathrm{true}}}
\newcommand{\tb}{B_{\mathrm{true}}}
\newcommand{\tc}{C_{\mathrm{true}}}
\newcommand{\td}{D_{\mathrm{true}}}

\newcommand{\du}[2]{u_{[#1,#2]}}
\newcommand{\dut}[1]{u_{[0,#1]}}
\newcommand{\dy}[2]{y_{[#1,#2]}}
\newcommand{\dyt}[1]{y_{[0,#1]}}
\newcommand{\dx}[2]{x_{[#1,#2]}}
\newcommand{\dxt}[1]{x_{[0,#1]}}
\newcommand{\dxtp}[1]{x_{[1,#1]}}

\newcommand{\dhxf}[2]{\hatx_{[0,T]}}

\newcommand{\nt}{n_{\mathrm{true}}}
\newcommand{\lt}{\ell_{\mathrm{true}}}

\newcommand{\nmt}{n_{\mathrm{min},t}}

\newcommand{\lmt}{\ell_{\mathrm{min},t}}

\newcommand{\syt}[1]{\calE_t(#1)}

\newcommand{\sylt}[2]{\calE_t(#1,#2)}

\newcommand{\sa}{\left[\begin{smallmatrix} A & B \\ C & D\end{smallmatrix}\right]}

\newcommand{\sai}{\left[\begin{smallmatrix} A_i & B_i \\ C_i & D_i\end{smallmatrix}\right]}

\newcommand{\struebig}{\left[\begin{array}{c|c} \ta & \tb \\\hline \tc & \td\end{array}\right]}

\newcommand{\noi}{\noindent}

\DeclareMathOperator{\lk}{lker}